\documentclass[12pt,a4paper,oneside]{amsart}
\usepackage[a4paper]{geometry}
\usepackage{mathptmx} 
\DeclareMathAlphabet{\mathcal}{OMS}{cmsy}{m}{n} 

\usepackage{amssymb,mathrsfs,enumerate}

\newtheorem{theorem}{Theorem}
\newtheorem{proposition}{Proposition}

\newtheorem{lemma}{Lemma}

\theoremstyle{definition}
\newtheorem*{remarks}{Remarks}
\newtheorem*{remark}{Remark}

\def\liebrack  {\ensuremath{[\,\cdot\, , \cdot\,]}}
\newcommand{\vertbar}{\>|\>}

\DeclareMathOperator{\ad}{\mathsf{ad}}
\DeclareMathOperator{\dcobound}{d}
\DeclareMathOperator{\End}{End}
\DeclareMathOperator{\Z}{Z}

\begin{document}

\title{Commutative post-Lie algebra structures on Kac--Moody algebras}

\author{Dietrich Burde}
\address{Universit\"at Wien, Wien, Austria}
\email{dietrich.burde@univie.ac.at}

\author{Pasha Zusmanovich}
\address{University of Ostrava, Ostrava, Czech Republic}
\email{pasha.zusmanovich@osu.cz}

\date{First written May 10, 2018; last minor revision May 6, 2019}
\thanks{Comm. Algebra, to appear; arXiv:1805.04267}

\begin{abstract}
We determine commutative post-Lie algebra structures on some 
infinite-dimensional Lie algebras. We show that all commutative post-Lie 
algebra structures on loop algebras are trivial. This extends the results for 
finite-dimensional perfect Lie algebras. Furthermore we show that all 
commutative post-Lie algebra structures on affine Kac--Moody Lie algebras are 
``almost trivial''.
\end{abstract}

\maketitle

\section{Introduction}

Recently there is a surge of interest in so-called post-Lie algebras and 
post-Lie algebra structures. One origin comes from the study of geometric 
structures on Lie groups, where post-Lie algebras arise as a common 
generalization of pre-Lie algebras \cite{HEL,KIM,SEG,BU5,BU19,BU24} and 
LR-algebras \cite{BU34, BU38}. Here pre-Lie algebras, also called left-symmetric
algebras, Vinberg algebras, or Koszul--Vinberg algebras, have been studied 
intensively before. For a survey, see \cite{BU24}. On the other hand, post-Lie 
algebras have been introduced by Vallette \cite{VAL} in $2007$ in connection 
with the homology of partition posets and the study of Koszul operads. Then they
were studied by several authors in various contexts, e.g., for algebraic operad triples \cite{LOD}, in connection with modified Yang--Baxter equations, 
Rota--Baxter operators, universal enveloping algebras, double Lie algebras, $R$-matrices, isospectral flows, Lie--Butcher series and 
many other topics \cite{BAG, ELM, GUB}.

Concerning post-Lie algebra structures on pairs of Lie algebras 
$(\mathfrak{g},\mathfrak{n})$, the existence question and the classification is
of particular interest. There have been many results obtained so far. In 
\cite{BU51} we introduced a special class of post-Lie algebra structures, namely
commutative ones. In this case, the two Lie algebras $\mathfrak{g}$ and 
$\mathfrak{n}$ coincide, and we obtain a bilinear commutative product satisfying
a certain compatibility condition with the Lie bracket, which can be considered
as a generalization of the left-symmetric identity (for a precise definition, 
see the introductory section below).

Commutative post-Lie algebra structures, CPA-structures in short, are much more
tractable, and we have obtained several existence and classification results 
\cite{BU51,BU52,BU57}. Among other things we proved in \cite{BU52} that any 
commutative post-Lie algebra structure on a finite-dimensional perfect Lie 
algebra over field of characteristic zero is trivial. Moreover we classified 
CPA-structures on certain classes of nilpotent Lie algebras. It is natural to 
study CPA-structures also for infinite-dimensional Lie algebras. In \cite{TAN} 
and \cite{TAY} this has been done already for the two-sided 
infinite-dimensional Witt algebra and some of its generalizations.

We want to continue these investigations in this paper. We will prove that 
CPA-structures on loop algebras are trivial, and ``almost trivial'' on 
Kac--Moody algebras; see Theorem \ref{th-2} and Theorem \ref{th-22} for exact 
formulations.

\section{Definitions, notations and conventions}

Let $A$ be a nonassociative algebra over a field $K$ in the sense of Schafer 
\cite{SCH}, with $K$-bilinear product $A\times A\rightarrow  A$, 
$(a,b)\mapsto ab$. We will assume that $K$ is an arbitrary field of 
characteristic different from $2$, if not said otherwise. Consider bilinear maps
$\varphi: A \times A \to A$ such that for any $a\in A$, the linear map 
$\varphi(a,\cdot): A \to A$ is a derivation of $A$. In other words, 
\begin{equation*}
\varphi(a,bc) = \varphi(a,b)c + b\varphi(a,c)
\end{equation*}
for any $a,b,c \in A$. The set of such bilinear maps forms a vector space which
will be denoted by $\mathcal D(A)$, and the subspace of such \emph{symmetric}
maps will be denoted by $\mathcal D_{comm}(A)$.

Recall that a \emph{commutative post-Lie algebra structure} (or, 
\emph{CPA-structure}) on a Lie algebra $L$ is a new binary multiplication 
$\cdot$ on $L$ which lies in $\mathcal D_{comm}(L)$, i.e., satisfying
$$
x \cdot y = y \cdot x
$$
and
$$
x \cdot [y,z] = [x \cdot y,z] - [x \cdot z,y]
$$
and which, additionally, satisfies the condition
$$
[x,y] \cdot z = x \cdot (y \cdot z) - y \cdot (x \cdot z)
$$
for any $x,y,z \in L$. When evaluating CPA-structures on various classes of Lie
algebras, it will be convenient to write this new multiplication as a bilinear
map $\varphi: L \times L \to L$, i.e., the previous three conditions are written
in the form
\begin{gather}
\varphi(x,y) = \varphi(y,x)  \label{eq-comm1}  \\
\varphi(x,[y,z]) = [\varphi(x,y),z] - [\varphi(x,z),y]  \label{eq-1} \\
\varphi([x,y],z) = \varphi(x,\varphi(y,z)) - \varphi(y,\varphi(x,z)) \label{eq-nonlin}
\end{gather}
for any $x,y,z \in L$.

For a Lie algebra $L$, $\Z(L)$ denotes the center of $L$. If $\Z(L) = 0$, then 
$L$ is called \emph{centerless}. If $[L,L] = L$, then $L$ is called 
\emph{perfect}.

All unadorned tensor products are over the base field $K$. The symbol $\oplus$ 
denotes the direct sum in the category of vector spaces.

\section{Twisted loop algebras}\label{sec-loop}

Given a Lie algebra $L$ and a commutative associative algebra $A$, the 
\emph{current Lie algebra} $L \otimes A$ carries a multiplication uniquely
defined by the formula
$$
[x \otimes a, y \otimes b] = [x,y] \otimes ab ,
$$
where $x,y \in L$, and $a,b \in A$. In the particular case $A = K[t,t^{-1}]$,
the Laurent polynomial algebra, we speak of the \emph{(untwisted) loop Lie 
algebra associated with $L$}.

Now let 
\begin{equation}\label{eq-grr}
L = \bigoplus_{\overline i \in \mathbb Z / n\mathbb Z} L_{\overline i}
\end{equation}
be a $\mathbb Z / n\mathbb Z$-graded Lie algebra, and consider the Lie algebra
\begin{equation}\label{eq-gr}
\widehat L = \bigoplus_{i \in \mathbb Z} (L_{i(mod\, n)} \otimes t^i) .
\end{equation}
This subalgebra of the loop Lie algebra $L \otimes K[t,t^{-1}]$ will be called a
\emph{twisted loop Lie algebra associated to the graded Lie algebra $L$}. The
untwisted case is formally included in the twisted one, when $n=1$ and the 
grading (\ref{eq-grr}) consists of the single zero component.

The direct sum (\ref{eq-gr}) is a $\mathbb Z$-grading, what will be crucial in 
what follows. More generally, let $L = \bigoplus_{g \in G} L_g$ be a Lie algebra
graded by an abelian group $G$. Then both vector spaces $\mathcal D(L)$ and 
$\mathcal D_{comm}(L)$ inherit a $G$-grading from $L$. Indeed, let us say that a
bilinear map $\varphi: L \times L \to L$ \emph{has degree $g \in G$}, or, 
symbolically, $\deg \varphi = g$, if $\varphi(L_h,L_f) \subseteq L_{h+f+g}$ for
any $h,f \in G$. Then:

\begin{proposition}\label{prop-33}
For an arbitrary $G$-graded Lie algebra $L$, we have
\begin{align*}
\mathcal D(L) & = 
\bigoplus_{g \in G} \{\varphi \in \mathcal D(L) \vertbar \deg \varphi = g\} , 
\\
\mathcal D_{comm}(L) & = 
\bigoplus_{g \in G} 
\{\varphi \in \mathcal D_{comm}(L) \vertbar \deg \varphi = g\} .
\end{align*}
\end{proposition}

The proof is standard, and follows almost verbatim the proof of a similar
statement for derivations. See, for example, Proposition 1.1 in \cite{FAR}.

In the sequel, it will be convenient to make use of the following auxiliary 
definition. Let us say that the Lie algebra $L$ \emph{satisfies the 
condition (*)} if any $\varphi \in \mathcal D_{comm}(L)$ such that 
$\varphi(x,\varphi(y,z)) = \varphi(y,\varphi(x,z))$ for any $x,y,z \in L$, 
vanishes.

\begin{proposition}\label{prop-p}
Let $L$ be a $\mathbb Z / n\mathbb Z$-graded Lie algebra satisfying the 
condition (*), and $\widehat L$ the twisted loop algebra associated to $L$. Then
there is a bijection between the sets of CPA-structures on $\widehat L$ and 
CPA-structures on $L$. Namely, any CPA-structure on $\widehat L$ is of the form 
\begin{equation}\label{eq-ij}
(x \otimes t^i, y \otimes t^j) \mapsto \varphi(x,y) \otimes t^{i+j} ,
\end{equation}
for any $x \in L_{i(mod\, n)}$, $y \in L_{j(mod\, n)}$ and $i,j \in \mathbb Z$,
where $\varphi$ is a CPA-structure on $L$.
\end{proposition}

\begin{proof}
Let $\Phi$ be a CPA-structure on $\widehat L$. By Proposition \ref{prop-33},
\begin{equation}\label{eq-sum1}
\Phi = \sum_\ell \Phi_\ell ,
\end{equation}
where $\Phi_\ell$ is an element of $\mathcal D_{comm}(\widehat L)$ of degree 
$\ell$, i.e. 
$$
\Phi_\ell(x \otimes t^i, y \otimes t^j) = 
\varphi_\ell(x,y) \otimes t^{i + j + \ell}
$$
for any $x \in L_{i(mod\, n)}$, $y \in L_{j(mod\, n)}$, and some bilinear map 
$\varphi_\ell: L \times L \to L$ of degree $\ell(mod\,n)$. The commutativity of 
$\Phi_\ell$ is equivalent to the commutativity of $\varphi_\ell$, and the 
condition (\ref{eq-1}) for $\Phi_\ell$ is equivalent to the same condition for 
$\varphi_\ell$, whence $\varphi_\ell \in \mathcal D_{comm}(L)$ for any $\ell$.

The condition (\ref{eq-nonlin}) for $\Phi$ is equivalent to
\begin{equation}\label{eq-ls}
\sum_\ell \varphi_\ell(x,[y,z]) \otimes t^{i + j + k + \ell} =
\sum_m \sum_s
\big(\varphi_m(x,\varphi_s(y,z)) - \varphi_s(y,\varphi_m(x,z))\big) 
\otimes t^{i + j + k + m + s}
\end{equation}
for any $x \in L_{i(mod\, n)}$, $y \in L_{j(mod\, n)}$, $z \in L_{k(mod\, n)}$,
and $i,j,k \in \mathbb Z$. Assume the sum (\ref{eq-sum1}) contains elements 
of positive degree, and let $N$ be the largest such degree. The maximal possible
degree of a summand at the left-hand side of the equality (\ref{eq-ls}) is 
$i + j + k + N$, while at the right-hand side the summand 
\begin{equation}\label{eq-e1}
\varphi_N(x,\varphi_N(y,z)) - \varphi_N(y,\varphi_N(x,z))
\end{equation}
has degree $i + j + k + 2N$, and all other summands have a smaller degree. 
Consequently, the expression (\ref{eq-e1}) vanishes for any $x$, $y$, $z$ 
belonging to arbitrary homogeneous components of $L$, and hence vanishes for any
$x,y,z \in L$. But then $\varphi_N$ vanishes, a contradiction. The same 
reasoning shows that the sum (\ref{eq-sum1}) does not contain summands of 
negative degree, and hence $\Phi$ is of degree zero, i.e. of the form 
(\ref{eq-ij}). In this situation, the condition (\ref{eq-nonlin}) for $\Phi$
is equivalent to the condition (\ref{eq-nonlin}) for $\varphi$.
\end{proof}

\begin{remark}
An analogous result can be obtained by replacing the Laurent polynomial algebra
by the (ordinary) polynomial algebra $K[t]$, or any similar graded 
polynomial-like algebra.
\end{remark}

\begin{lemma}\label{cor-3}
Let $L$ be a Lie algebra such that:
\begin{enumerate}[\upshape(i)]
\item\label{it-c1}
$L$ is centerless;
\item\label{it-c2} 
all derivations of $L$ are inner;
\item\label{it-c3}
any linear map $\omega$ from $L$ to an abelian subalgebra of $L$, satisfying the
condition 
$$
[\omega(x),y] + [x,\omega(y)] = 0
$$
for any $x,y\in L$, vanishes.
\end{enumerate}
Then $L$ satisfies the condition (*).
\end{lemma}

\begin{proof}
Since all derivations of $L$ are inner, any map 
$\varphi \in \mathcal D_{comm}(L)$ has the form $\varphi(x,y) = [y,\omega(x)]$ 
for $x,y\in L$ and some linear map $\omega: L \to L$. Since $\varphi$ is
symmetric, we have $[\omega(x),y] + [x,\omega(y)] = 0$. The condition
$\varphi(x,\varphi(y,z)) = \varphi(y,\varphi(x,z))$ is equivalent then to 
$[[z,\omega(y)],\omega(x)] = [[z,\omega(x)],\omega(y)]$ what, together with the
Jacobi identity and the fact that $L$ is centerless, implies 
$[\omega(x),\omega(y)] = 0$, i.e. $\omega(L)$ is an abelian subalgebra in $L$.
Now (\ref{it-c3}) implies that $\omega$, and hence $\varphi$, vanishes, 
i.e. $L$ satisfies the condition (*).
\end{proof}

Now we can prove one of our main results.

\begin{theorem}\label{th-2}
Let
\[ 
\mathfrak g = \bigoplus_{i\in \mathbb Z / n\mathbb Z} \mathfrak g_i
\] 
be a $\mathbb Z / n\mathbb Z$-graded simple finite-dimensional complex Lie 
algebra. Then any CPA-structure on the associated twisted loop algebra
$\widehat{\mathfrak g}$ vanishes.
\end{theorem}

\begin{proof}
The Lie algebra $\mathfrak g$ satisfies the conditions of Lemma \ref{cor-3}.
Indeed, (\ref{it-c1}) is evident and (\ref{it-c2}) is well-known. According to
Lemma 6.1 in \cite{LEL}, \emph{any} linear map  
$\omega: \mathfrak g \to \mathfrak g$ satisfying 
$[\omega(x),y] + [x,\omega(y)] = 0$ for any $x,y \in \mathfrak g$, vanishes, 
thus (\ref{it-c3}) is satisfied. Therefore, $\mathfrak g$ satisfies the 
condition (*), and by Proposition \ref{prop-p} the set of CPA-structures on 
$\widehat{\mathfrak g}$ is in bijection with the set of CPA-structures on 
$\mathfrak g$. But any CPA-structure on $\mathfrak g$ vanishes according to 
Proposition 5.4 of \cite{BU51}, or Proposition 3.1 of \cite{BU52}.
\end{proof}

\begin{remarks}\hfill

(i)
On practice, only the cases $n=1$ (the untwisted case) and $n=2,3$ (the twisted
case) may occur; see Chapter $8$ in \cite{KAC}.

(ii)
According to Theorem $3.3$ of \cite{BU52}, CPA-structures vanish not only on 
$\mathfrak g$, but on any perfect finite-dimensional Lie algebra over a field of characteristic zero. The twisted loop algebra $\widehat{\mathfrak g}$ is 
perfect, but not finite-dimensional, so Theorem \ref{th-2} is not covered, at 
least in a straightforward way, by that result.
\end{remarks}

\section{Digression: graded Lie algebras, Witt algebras, and current algebras}

This section contains some comments on and alternative approaches to the results
of the previous section. The proofs are omitted, and the results stated here 
will be not used in the next section.

In the previous section we took advantage of the graded structure of twisted
loop algebras, which, being coupled with the nonlinear condition 
(\ref{eq-nonlin}), implies a strong restriction on CPA-structures. One of the 
possible generalizations along these line is:

\begin{proposition}\label{prop-gr}
Let $L$ be a $\mathbb Z$-graded Lie algebra $L = \bigoplus_{i\in \mathbb Z} L_i$
satisfying the condition (*). Then any CPA-structure $\varphi$ on $L$ is of 
degree $0$: $\varphi(L_i,L_j) \subseteq L_{i+j}$ for any $i,j \in \mathbb Z$.
\end{proposition}

The proof is straightforward and is similar to the proof of 
Proposition \ref{prop-p}. In the situation of twisted loop algebras it was more 
convenient to use a somewhat more specific Proposition \ref{prop-p}, and not 
this general result. 

However, Proposition \ref{prop-gr} can be used to establish the vanishing of 
CPA-structures on Witt algebras in a somewhat different way than this was done 
in \cite{TAN}. Recall that the Witt algebras are defined as Lie algebras over a
field $K$ of characteristic zero, having a set of basis vectors $\{e_i\}$ with multiplication
$$
[e_i,e_j] = (j-i) e_{i+j} .
$$
Depending on whether the indices run over all integers, or over integers 
$\ge -1$, we get the two-sided or one-sided Witt algebra, respectively.

\begin{theorem}
Any CPA-structure on a Witt algebra (one- or two-sided) vanishes.
\end{theorem}

The proof is obtained by an easy combination of reasonings as in the proof of 
Corollary \ref{cor-3}, the facts that all derivations of a Witt algebra are 
inner, and all abelian subalgebras are one-dimensional, and 
Proposition \ref{prop-gr}.

Picking up another thread in \S \ref{sec-loop}, let us outline an alternative
approach to the proof of Theorem \ref{th-2}; it takes advantage of the fact that
one of the defining conditions of the post-Lie algebra is linear, and employs a
linear-algebraic technique from \cite{ZUS}, used earlier in \cite{ZUS} and 
\cite{ZUS2} to describe other kinds of linear structures on current Lie algebras, such as derivations, low-degree 
cohomology, Poisson structures, etc., in terms of some invariants of the tensor
factors. 

Before we formulate the corresponding statement, a few definitions are in order.
Recall that the \emph{centroid} of a Lie algebra $L$ is the space of linear maps
$\varphi: L \to L$ commuting with inner derivations of $L$, i.e. 
$$
\varphi([x,y]) = [\varphi(x),y]
$$
for any $x,y\in L$. A Lie algebra is called \emph{central}, if its centroid
coincides with multiplications by the elements of the base field.

The set of all bilinear maps $\varphi: L \times L \to L$ such that for any 
$x \in L$, the linear map $\varphi(x,\cdot): L \to L$ belongs to the centroid of
$L$, i.e.,
$$
\varphi(x,[y,z]) = [\varphi(x,y),z]
$$
for any $x,y,z\in L$, forms a vector space denoted by $\mathcal C(L)$.

For a vector space $V$, $\End(V)$ denotes the set of all linear maps $V \to V$.

\begin{proposition}\label{prop-1}
Let $L$ be a centerless Lie algebra, $A$ an associative commutative algebra with
unit, and one of $L$, $A$ is finite-dimensional. Then 
\begin{equation}\label{eq-i}
\mathcal D(L \otimes A) \simeq 
\Big(\mathcal D(L) \otimes \End(A)\Big) \oplus 
\Big(\mathcal C(L) \otimes \mathcal D(A)\Big) .
\end{equation}
Each element of $\mathcal D(L \otimes A)$ can be written as a sum of maps of 
the form $\varphi \otimes \alpha$, where $\varphi: L \times L \to L$ and 
$\alpha: A \times A \to A$ are bilinear maps, of the two kinds:
\begin{enumerate}[\upshape(i)]
\item 
$\varphi \in \mathcal D(L)$, and $\alpha(a,b) = \beta(a)b$ for any $a,b \in A$ 
and some linear map $\beta: A \to A$;
\item
$\varphi \in \mathcal C(L)$, and $\alpha \in \mathcal D(A)$.
\end{enumerate}
\end{proposition}

The proof is very similar to the proof of the formula from Corollary 2.2 in 
\cite{ZUS} expressing derivations of a current Lie algebra in terms of its tensor factors. The condition that $L$ is centerless 
is not crucial and can be removed at the expense of more laborious computations and cumbersome formulas.

Similarly, we have:

\begin{proposition}\label{prop-2}
In the setup of Proposition \ref{prop-1}, assume additionally that $L$ is
central. Then
$$
\mathcal D_{comm}(L \otimes A) \simeq \mathcal D_{comm}(L) \otimes A .
$$
Each element of $\mathcal D_{comm} (L \otimes A)$ can be written as the sum of
maps of the form $\varphi \otimes \alpha$, where 
$\varphi \in \mathcal D_{comm}(L)$, and $\alpha: A \times A \to A$ is a bilinear
map of the form $\alpha(a,b) = abu$ for some $u \in A$.
\end{proposition}

Then, using Proposition \ref{prop-2}, we may impose on 
$\mathcal D_{comm}(L \otimes A)$ the additional condition (\ref{eq-nonlin}) to 
try to get an analogous formula for the set of CPA-structures on $L \otimes A$.
However, the nonlinearity of (\ref{eq-nonlin}) makes the task much more 
difficult, and we seemingly have to abandon the generality of 
Propositions \ref{prop-1} and \ref{prop-2}, and assume $A$ to be a more or less
concrete algebra. This provides a somewhat alternative way to the results of 
\S \ref{sec-loop}, at least in the nontwisted case.

\section{Kac--Moody algebras}

We use the well-known realization of untwisted and twisted affine Kac--Moody 
algebras as extensions of current Lie algebras by a derivation and a central 
element:
\begin{equation}\label{eq-a}
\widehat{\mathfrak g} \oplus 
\mathbb C t\frac{\dcobound}{\dcobound t} \oplus \mathbb Cz \>, 
\end{equation}
where $\widehat{\mathfrak g}$ denotes, as previously, the twisted loop algebra 
associated to a $\mathbb Z / n\mathbb Z$-graded simple 
finite-di\-men\-si\-o\-nal complex Lie algebra $\mathfrak g$. The Euler 
derivation $t\frac{\dcobound}{\dcobound t}$ acts on the current Lie algebra 
$\mathfrak g \otimes \mathbb C[t,t^{-1}]$, and, by restriction, on 
$\widehat g \subseteq \mathfrak g \otimes \mathbb C[t,t^{-1}]$, as the 
derivation of the second tensor factor, i.e.
$$
[x \otimes t^i, t\frac{\dcobound}{\dcobound t}] = i x \otimes t^i ,
$$
where $x \in L$ and $i \in \mathbb Z$. The element $z$ is central, and 
multiplication on $\widehat{\mathfrak g}$ is twisted by the $\mathbb C z$-valued
``Kac--Moody cocycle'' (whose concrete form is not important for us here).

\begin{lemma}\label{lemma-1}
Let $L$ be a perfect centerless Lie algebra such that any CPA-structure on $L$
vanishes, $D$ an outer derivation of $L$, and $\mathscr L = L \oplus KD$ be the
semidirect sum with $D$ acting on $L$. Then any CPA-structure on $\mathscr L$ 
vanishes.
\end{lemma}

\begin{proof}
Let $\Phi$ be a CPA-structure on $\mathscr L$. We may write 
\begin{alignat*}{2}
&\Phi(x,y) &\>\>=\>\>& \varphi(x,y) + \lambda(x,y)D \\
&\Phi(x,D) &\>\>=\>\>& \psi(x) + \mu(x) D           \\
&\Phi(D,D) &\>\>=\>\>& a + \eta D
\end{alignat*}
where $x,y \in L$, for some (bi)linear maps $\varphi: L \times L \to L$, 
$\lambda: L \times L \to K$, $\psi: L \to L$, $\mu: L \to K$, and $a \in L$, 
$\eta \in K$. Then the symmetricity of $\Phi$ implies the symmetricity of 
$\varphi$, and the condition (\ref{eq-1}) for $\Phi$, written for an arbitrary 
triple $x,y,z\in L$, is equivalent to
$$
\varphi(x,[y,z]) = [\varphi(x,y),z] - [\varphi(x,z),y] 
                 + \lambda(x,y)D(z) - \lambda(x,z)D(y) ,
$$
and to the condition $\lambda(L,[L,L]) = 0$. But since $L$ is perfect, $\lambda$
vanishes, and hence $\varphi \in \mathcal D_{comm}(L)$. Imposing on $\Phi$ the 
condition (\ref{eq-nonlin}) leads to the conclusion that the restriction 
$\varphi = \Phi|_L$ is a CPA-structure on $L$, and hence vanishes.

The condition (\ref{eq-1}) for $\Phi$, written for the triple $D,x,y$, where 
$x,y\in L$, implies $\mu([L,L]) = 0$, whence $\mu$ vanishes. Then writing the 
same condition for the triple $x,D,y$, and taking into account all the vanishing
conditions obtained so far, we get $\psi(L) \subseteq Z(L) = 0$, whence $\psi$
vanishes.

Finally, the condition (\ref{eq-1}) for $\Phi$, written for the triple 
$D,D,x\in L$, yields $\eta D = \ad a$. But since $D$ is an outer derivation, $\eta = 0$, 
$a \in Z(L) = 0$, and $\Phi$ vanishes identically.
\end{proof}

\begin{lemma}\label{lemma-2}
Let $L$ be a centerless Lie algebra such that any CPA-structure on $L$ vanishes,
and $\mathscr L = L \oplus Kz$ a nontrivial one-dimensional central extension of
$L$. Then the set of CPA-structures on $\mathscr L$ consists of $Kz$-valued 
symmetric bilinear maps vanishing whenever one of the arguments belongs to 
$[L,L] \oplus Kz$.
\end{lemma}

\begin{proof}
Write the Lie bracket on $\mathscr L$ as $\{x,y\} = [x,y] + \xi(x,y)z$, where 
$x,y \in L$, $\liebrack$ is a Lie bracket on $L$, and $\xi$ is a (nontrivial) 
$2$-cocycle on $L$.

Let $\Phi$ be a CPA-structure on $\mathscr L$. Similarly with the proof of 
Lemma \ref{lemma-1}, we may write
\begin{alignat*}{2}
&\Phi(x,y) &\>\>=\>\>& \varphi(x,y) + \lambda(x,y)z \\
&\Phi(x,z) &\>\>=\>\>& \psi(x) + \mu(x) z           \\
&\Phi(z,z) &\>\>=\>\>& a + \eta z
\end{alignat*}
where $x,y \in L$, and all the maps and elements occurring at the right-hand 
side have the same meaning as in the proof of Lemma \ref{lemma-1}.

The symmetricity of $\Phi$ implies the symmetricity of $\varphi$ and $\lambda$.
The condition (\ref{eq-1}) for $\Phi$, written for the triple $x,y,z$, where 
$x,y\in L$, yields $\psi(L) \subseteq Z(L) = 0$, whence $\psi$ vanishes. Then the same condition written for the triple $z,x,y$, yields $a=0$ and 
$\eta\xi(x,y) = -\mu([x,y])$. But since $\xi$ is a nontrivial (i.e., not 
equal to a coboundary) cocycle, the latter equality implies $\eta = 0$.

Further, the condition (\ref{eq-1}) for $\Phi$, written for the triple 
$x,y,t \in L$, yields 
$$
\varphi(x,[y,t]) = [\varphi(x,y),t] - [\varphi(x,t),y] ,
$$
i.e., $\varphi \in \mathcal D_{comm}(L)$, and
\begin{equation}\label{eq-mu}
\lambda(x,[y,t]) + \xi(y,t)\mu(x) = 
\xi(\varphi(x,y),t) - \xi(\varphi(x,t),y) .
\end{equation}

The condition (\ref{eq-nonlin}) for $\Phi$, written for the triple 
$x,y,t\in L$, yields
$$
\varphi([x,y],t) = \varphi(x, \varphi(y,t)) - \varphi(y, \varphi(x,t)) ,
$$
i.e., $\varphi$ is a CPA-structure on $L$, whence $\varphi$ vanishes.

The condition (\ref{eq-nonlin}) for $\Phi$, written for the triple $z,x,x$, 
where $x \in L$, yields $\mu(x)^2 = 0$, whence $\mu$ vanishes. 

To summarize: the only nonzero values of $\Phi$ are given by 
$\Phi(x,y) = \lambda(x,y)z$, where $x,y\in L$. Moreover, (\ref{eq-mu}) now
implies $\lambda(L,[L,L]) = 0$. These are exactly the maps as specified in the 
statement of the lemma. It is straightforward to check that conversely, any
such map is a CPA-structure on $\mathscr L$.
\end{proof}

\begin{theorem}\label{th-22}
CPA-structures on an affine Kac--Moody algebra written in the form (\ref{eq-a}),
form a one-dimensional vector space, spanned by the map:
\begin{gather*}
\left( t \frac{\dcobound}{\dcobound t}, t \frac{\dcobound}{\dcobound t}\right) 
\mapsto z 
\\
\text{all other pairs of basis vectors} \mapsto 0 .
\end{gather*}
\end{theorem}

\begin{proof}
As by Theorem \ref{th-2} any CPA-structure on $\widehat{\mathfrak g}$ vanishes,
and the Euler derivation $t \frac{\dcobound}{\dcobound t}$ (in fact, any nonzero
derivation of the Laurent polynomials extended to the loop algebra 
$\widehat{\mathfrak g}$) is an outer derivation of $\widehat{\mathfrak g}$, 
Lemma \ref{lemma-1} implies that any CPA-structure on the semidirect sum 
$\widehat{\mathfrak g} \oplus \mathbb C t \frac{\dcobound}{\dcobound t}$ 
vanishes. This allows, in its turn, to apply Lemma \ref{lemma-2} to 
$L = \widehat{\mathfrak g} \oplus \mathbb C t \frac{\dcobound}{\dcobound t}$.
Since $[L,L] = \widehat{\mathfrak g}$, Lemma \ref{lemma-2} yields the desired
result.
\end{proof}

\section*{Acknowledgement}

Thanks are due to the referee for helpful remarks which improved the 
presentation. Dietrich Burde is supported by the Austrian Science Foun\-da\-tion FWF, grant 
P28079 and grant I3248.

\end{document}